\documentclass[11pt, oneside]{amsart}

\usepackage{geometry,graphicx,fancybox}
\usepackage{enumerate,amsmath,amssymb,amsthm}
\usepackage{comment}

\usepackage{mathtools}
\usepackage{enumitem}
\usepackage{bbold}
\usepackage{mathrsfs}
\usepackage[colorlinks,linkcolor=blue]{hyperref}
\usepackage{dsfont}
\usepackage{pgf,tikz}
\usetikzlibrary{arrows}
\usetikzlibrary[patterns]

\theoremstyle{plain}

\usepackage[kerning=true]{microtype} 
\usepackage{graphicx}
\usepackage{wrapfig}
\usepackage{multirow}

\newtheorem{theorem}{Theorem}

\newtheorem{proposition}[theorem]{Proposition}
\newtheorem{definition}[theorem]{Definition}

\theoremstyle{remark}
\theoremstyle{remark}

\theoremstyle{definition}
\newtheorem{remark}[theorem]{Remark}

\DeclareMathOperator{\area}{Area}

\DeclareMathOperator{\sys}{sys}

\newcommand{\abs}[1]{\left\vert#1\right\vert}
\newcommand{\set}[1]{\left\{#1\right\}}
\newcommand{\crochet}[1]{\left[#1\right]}

\newcommand{\fonction}[5]{\begin{array}{cccl}           
#1: & #2 & \longrightarrow & #3 \\
    & #4 & \longmapsto & #5 \end{array}}

\title[Finsler systolic ratio on the 2-sphere]{A Finsler counterexample to the Croke conjecture for the systolic ratio on the 2-sphere}

\author{Guillaume Buro and Louis Merlin}

\address{G. Buro, École Polytechnique Fédérale de Lausanne, Suisse.}

\email{guillaume.buro@epfl.ch}

\address{L. Merlin, Rheinisch-Westfälische Technische Hochschule (RWTH), Aachen, Germany}

\email{louis.merlin@hotmail.fr}

\thanks{The first author acknowledges support by SNSF grant \textit{Metric geometry and Finsler structures of low regularity} number 200021L-175985, the second by RWTH University. Both authors would like to thank Prof. Umberto Hryniewicz and Prof. Marc Troyanov for the interest that they showed on this project and for interesting discussions.}

\subjclass[2010]{37D50, 53B40, 53C22}

\begin{document}

\begin{abstract}
We exhibit a Finsler metric on the $2$-sphere whose systolic (Holmes-Thompson) ratio is $\frac{4\pi}{3}$. This is bigger than the conjectured maximal Riemannian systolic ratio of $2\sqrt{3}$ achieved by the Calabi-Croke metric. The construction of the Finsler metric is heavily inspired by \cite{cossarinisabourau}.
\end{abstract}

\maketitle

\section{Introduction.}

As a general rule of thumb, Finsler metrics produce sharper geometric inequalities than the special case of Riemannian metrics. Evidences for that have been given in different situations: for the minimal entropy problem, compare e.g \cite{bcg2} (Riemannian) and \cite{verovicminimalentropy} (Finslerian), Finsler rational counterexamples to the minimal filling conjecture \cite{buragoivanovfilling} or (more closely related to the present text), the disk of the least area for a given radius is Finsler, compare \cite{cclw} (Riemannian) and \cite{cossarinisabourau} (Finslerian).

In this paper, we confirm this phenomenon in the context of the systolic geometry of the 2-dimensional sphere.

We first consider a Riemannian 2-sphere $(\mathbb{S}^2,g)$. Its \textit{systole}, denoted $\sys (g)$, is the length of the shortest non constant periodic geodesic \footnote{This definition is adapted to the case of the sphere; when the manifold has topology it is classical to consider the least length of \textit{non contractible} geodesics}. As a length, the systole is not scale invariant and it is customary to normalize the systole by the area. Indeed, the ratio
\[\frac{\sys^2(g)}{\area(g)},\]
called the \textit{systolic ratio} is scale invariant and, by a result of \cite{crokesystolbounded}, is bounded on the set of Riemannian metrics on $\mathbb{S}^2$, thus opening the challenging problem of finding the best systolic ratio $\sys (\mathbb{S}^2)$ on the $2$-sphere:
\[\sys(\mathbb{S}^2)=\sup_g \frac{\sys^2(g)}{\area(g)}.\]
It is conjectured in \cite{crokesystolbounded} that the extremal metric is the Calabi-Croke doubled triangle: consider two flat equilateral triangles whose side lengths are 1, glued along the edges. The resulting surface is homeomorphic to a sphere, its area is $\frac{\sqrt{3}}{2}$ (twice the area of a triangle) and the systole is achieved by the curves running through the heights of the triangles, their lengths are $\sqrt{3}$, so that the systolic ratio is $2\sqrt{3}$. The Calabi-Croke sphere is at least a local maximum for the systolic ratio \cite{balacheffcalabicroke}, even for a local variation among Finsler metrics \cite{sabouraucalabicroke}.

The main result of this paper is that we can go past this value of $2\sqrt{3}$ by considering Finsler metrics and we prove the following

\begin{theorem}\label{thm:main}
There exists a smooth Finsler metric on $\mathbb{S}^2$ with six singularities whose systolic ratio is arbitrarily close to $\frac{4\pi}{3}$.\footnote{Indeed $3.464\cdots=2\sqrt{3} < \frac{4\pi}{3}= 4.188\cdots$.}
\end{theorem}

\begin{remark}
The metric appearing in the statement of theorem \ref{thm:main} is only defined on $\mathbb{S}^2$ minus 6 points. We do not know how to remove singularities (nor if it is possible to do so). We discuss the regularity of this metric in Section \ref{sec:regularity}.
\end{remark}

This result gives a hint on the methods relevant to attack the Croke conjecture. Since the conjecture is already false for Finsler manifolds, none of the generic metric techniques may apply and the technology for the proof must stay inside classical Riemannian geometry.

In the preliminary section \ref{sec:preliminaries}, we recall the definition of Finsler surfaces and their Holmes-Thompson volumes, in section \ref{sec:cossarinisaboura}, we present the construction of the Cossarini-Sabourau Finlser disk and in section \ref{sec:conclusion}, we perform a construction similar to the Calabi-Croke metric with the Cossarini-Sabourau disk and show that the metric we obtain is approximated by Finsler metrics whose systolic ratio converge to $\frac{4\pi}{3}$.

\section{Finsler surfaces.}\label{sec:preliminaries}

Finsler metrics are a relaxation of Riemannian metrics, similar than enlarging the class of ellipsoids by arbitrary convex sets. More precisely:

\begin{definition}
Let $S$ be a smooth surface. A (smooth) Finsler metric on $S$ is a positive function $F:TS\rightarrow \mathbb{R}_+$ satisfying the following conditions:
\begin{itemize}
\item The restriction of $F$ to any $T_xS$ is a symmetric norm.
\item $F$ is smooth outside the zero section.
\item For any $v\in TS$ the Hessian of $F^2$ at $v$ is positive definite.
\end{itemize}
\end{definition}

It follows from the definition that the unit ball
\[B_x=\set{v\in T_xS\;\;F(x,v)\leqslant 1}\]
is a strict convex set.

The systolic ratio features the area of the metric, canonically defined for a Riemannian metric but different choices of area measures coexist for a Finsler metric. Here, we consider the so-called Holmes-Thompson area: its symplectic nature makes it easily computable with Blashke and Santal\'o formulas.

In order to define the Holmes-Thompson area, we denote by $F_x$ the norm on $T_xS$ given by $F$ and by $F_x^*$ the dual norm on $T_x^*S$ (whose unit ball is the dual convex set $B_x^*$ of $B_x$).

The cotangent space naturally carries a symplectic structure expressed in coordinates $(x_1,x_2,\xi_1,\xi_2)$ by
\[\omega=d\xi_1\wedge dx_1 + d\xi_2\wedge dx_2.\]
Note that $\Omega=\frac{1}{2}\omega\wedge\omega$ is a volume form on $T^*S$. The Holmes-Thompson area is the push-forward onto $S$ by the canonical projection of this volume form restricted to the unit co-ball bundle: 

\begin{definition}
Let $A$ be a Borel set in $S$. Its Holmes-Thompson area is given by
\[\area(A)=\frac{1}{\pi}\int_{\cup_{x\in A}B_x^*S}\Omega.\]
\end{definition}
The constant $\frac{1}{\pi}$ is a canonical normalization constant, so that for instance the area of a Euclidean unit ball is exactly $\pi$. Throughout this text is hidden the fact that the Holmes-Thompson area is computed by Blaschke and Santalo formulas, as in \cite[Section 3]{cossarinisabourau}. We consider again this symplectic measure in Section \ref{sec:regularity}.

\section{The Cossarini-Sabourau Finsler disk.}\label{sec:cossarinisaboura}

In this section we recall the construction of the Cossarini-Sabourau Finsler disk (see \cite[Section 11]{cossarinisabourau} for details).

We consider a regular hexagon $H$, centered at $0$ in the complex Euclidean plane and whose vertices are located at $\frac{2}{\sqrt{3}}e^{ik\pi /3}$ for $k=0,\cdots 5$. The constant $\frac{2}{\sqrt{3}}$ is chosen in such a way that the metric we are going to describe turns the hexagon into a unit ball. To define this metric, we denote by $L_k$ the lines $\set{te^{i\frac{\pi}{6}+k\frac{2\pi}{3}}\;\;t\in\mathbb{R}}$, by $\overline{L_k}$ the unit segment $\set{te^{i\frac{\pi}{6}+k\frac{2\pi}{3}}\;\;t\in\crochet{0,1}}$, by $\pi_k$ the orthogonal projection onto $L_k$ and by $\overline{\pi_k}$ the trace of the projection $\pi_k$ on $\overline{L_k}\subset L_k$. Finally, we denote by $\abs{\cdot}$ the usual Lebesgue measure on any of the lines $L_k$. See figure \ref{fig:hexagon}.

\definecolor{ffqqqq}{rgb}{1,0,0}
\definecolor{zzzzff}{rgb}{0.6,0.6,1}
\begin{figure}[!h]
\begin{tikzpicture}[line cap=round,line join=round,x=1.0cm,y=1.0cm]
\clip(-2.12,1.28) rectangle (10.48,9.96);
\fill[line width=0.4pt,color=zzzzff,fill=zzzzff,fill opacity=0.1] (1.7,2.04) -- (5.74,2.04) -- (7.76,5.54) -- (5.74,9.04) -- (1.7,9.04) -- (-0.32,5.54) -- cycle;
\draw [line width=0.4pt,color=zzzzff] (1.7,2.04)-- (5.74,2.04);
\draw [line width=0.4pt,color=zzzzff] (5.74,2.04)-- (7.76,5.54);
\draw [line width=0.4pt,color=zzzzff] (7.76,5.54)-- (5.74,9.04);
\draw [line width=0.4pt,color=zzzzff] (5.74,9.04)-- (1.7,9.04);
\draw [line width=0.4pt,color=zzzzff] (1.7,9.04)-- (-0.32,5.54);
\draw [line width=0.4pt,color=zzzzff] (-0.32,5.54)-- (1.7,2.04);
\draw [line width=0.4pt] (3.72,5.54)-- (-0.32,5.54);
\draw [line width=0.4pt] (3.72,5.54)-- (1.7,9.04);
\draw [line width=0.4pt] (3.72,5.54)-- (5.74,9.04);
\draw [line width=0.4pt] (3.72,5.54)-- (7.76,5.54);
\draw [line width=0.4pt] (3.72,5.54)-- (5.74,2.04);
\draw [line width=0.4pt] (3.72,5.54)-- (1.7,2.04);
\draw [line width=1.2pt,color=ffqqqq] (0.69,7.29)-- (3.72,5.54);
\draw [line width=1.2pt,color=ffqqqq] (3.72,5.54)-- (6.75,7.29);
\draw [line width=1.2pt,color=ffqqqq] (3.72,5.54)-- (3.72,2.04);
\draw (0.52,4.78) node[anchor=north west] {$H$};
\draw (3.88,2.93) node[anchor=north west] {$\overline{L_2}$};
\draw (5.92,6.73) node[anchor=north west] {$\overline{L_0}$};
\draw (0.74,6.99) node[anchor=north west] {$\overline{L_1}$};
\end{tikzpicture}
\caption{Geometry of the distance $d_H$.}
\label{fig:hexagon}
\end{figure}
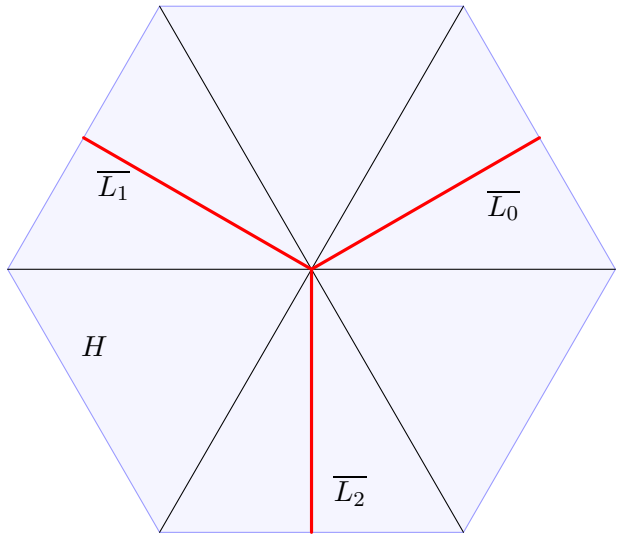

\begin{definition}
\begin{enumerate}
\item Let $\gamma:\crochet{0,1}\to H$ be a piecewise smooth curve. We define its length by
\[l(\gamma)=\sum_{k=0,1,2}\abs{\overline{\pi_k}(\gamma(\crochet{0,1}))}.\] 
\item The pseudo-metric on $H$ is the one associated to this length structure:
\[d_H(x,y)=\inf_\gamma l(\gamma),\]
where $\inf$ is taken on all piecewise smooth curves joining $x$ to $y$.
\end{enumerate}
\end{definition}

\begin{remark}
We actually rescaled the Cossarini-Sabourau metric by a factor 4. This doesn't change the systolic ratio.
\end{remark}

We are now interested in the metric properties of $(H,d_H)$. Note first that $d_H$ is not an actual metric because there are different points at distance 0. Indeed a straight line crossing one of the $\overline{L_i}$ orthogonally and sufficiently small has length 0. We could consider the quotient of $H$ by points at distance 0 but the description of the geodesic flow is easier in the hexagon.

From \cite{cossarinisabourau}, we sumarize the relevant properties (a sketch of the argument is presented in Section \ref{sec:regularity}).

\begin{proposition}{\cite[Lemma 11.5, Poposition 11.7, Remark 11.3]{cossarinisabourau}}\label{prop:metricproperties}
\begin{enumerate}
\item The pseudo-metric space $(H,d_H)$ is a limit for the $L^\infty$ topology of Finsler metric spaces.
\item Through $L^\infty$ approximations by Finsler metrics, the Holmes-Thompson area converges to $\frac{6}{\pi}$.
\item The geodesics of the Finlser approximations are the Euclidean straight lines.
\end{enumerate}
\end{proposition}

\begin{remark}\label{rmk:projectivity}
The geodesic flow of $(H,d_H)$ is not well defined: $0$ is a branching point and there is no uniqueness of geodesics.

Approximating $(H,d_H)$ by projective Finsler spaces has the double interest of transforming the pseudo-metric to an actual metric and also selecting the straight lines as geodesics.

In the rest of this paper, we consider only the restricted geodesic flow on $(H,d_H)$ which consists in straight lines in $H$. Every other geodesic is irrelevant to compute the systolic ratio.
\end{remark}

\section{The Calabi-Croke Cossarini-Sabourau Finsler sphere and its systolic ratio.}\label{sec:conclusion}

We now perform a construction similar to the Calabi-Croke metric on $\mathbb{S}^2$: we glue two Cossarini-Sabourau disks along the edges. The resulting metric in not quite Finsler but is approximated by the glueing of two Finlser disks. Since the approximation preserves the systolic ratio, we reason with the limiting (non Finlser) metric that we denote $(\mathbb{S}^2,d_H)$.

Note that the metric is not defined at the vertices of the hexagons. Consequently, there is no geodesic passing through a vertex. Moreover the metric is (a priori) only continuous on the edges (indeed the Finsler metric is well defined on the edges of the Cossarini-Sabourau hexagon and, since we took two copies of the same hexagon, there is a well-defined metric on the glueing edges and the resulting metric is continuous). See Section \ref{sec:regularity} for a comment on the regularity. 

We denote by $p$ the canonical projection $p:\mathbb{S}^2\rightarrow H$. Directly following from the fact that (restricted) geodesics are straight lines, we get a description of the geodesic flow in $H$.

\begin{proposition}
The image of restricted geodesics through $p$ consists in billiard trajectories inside $H$.
\end{proposition}

\begin{proof}
Inside a hexagon, geodesics are the straight lines. When a geodesic hits an edge, it can only do so transversaly, since there is no geodesic passing through a vertex, and is continued in a straight line on the other hexagon with a direction given by the velocity of the trajectory on the edge. Moreover, on an edge, the direction of a geodesic in one hexagon is mapped through $p$ to the reflected direction (according to Descartes law of reflection).
\end{proof}

The rest of this paper is devoted to the computation of the systole. Precisely:

\begin{proposition}\label{prop:systole}
The systole of $(\mathbb{S}^2,d_H)$ for the restricted geodesic flow is $4$ and is achieved for instance by a curve running twice through a height of $H$.
\end{proposition}

\begin{proof}
Subsequently the most important property of the hexagon is that it tiles the plane, allowing to precisely describe the geodesic flow on $(\mathbb{S}^2,d_H)$. This procedure is classical but we recall it for completeness (see \cite{gutkinbilliardsurvey}).


Inside $H$, there is no trapped geodesic so we may assume that the starting point of a geodesic belongs to an edge of $H$. Without loss of generality, this edge is supported on the $x$ axis in $\mathbb{R}^2$ and contains 0. In each hexagon, there is two different types of equilaleral triangles and we may also assume the the initial edge belongs to a triangle carrying a metric $l^1$ (such triangles are the ones containing none of the segment $\overline{L_i}$). Indeed, any point in $H$ is connected to a point in a $l^1$-type triangle by a geodesic of length 0. Hence we don't change the systole by assuming that the initial point of every geodesic lies on the edge of a $l^1$-type triangle. 

Recall that we chose the normalization in such a way that the length of the edges of $H$ are $\frac{2}{\sqrt{3}}$ (the heights have length 1).

Let
\[e_1=\begin{pmatrix}
\frac{2}{\sqrt{3}}\\1
\end{pmatrix}
\;\;\mbox{ and }\;\;e_2=\begin{pmatrix}
0\\2
\end{pmatrix}.
\]
We tile the plane by hexagons so that, to every point of the lattice $\mathbb{Z}e_1\oplus\mathbb{Z}e_2$, corresponds a unique hexagon. To each hexagon is attributed a name according to each point of the lattice it corresponds. For instance the original hexagon $H$ is named $(0,0)$, the one just above is $(0,1)$, the one image of the original one by $e_1$ is $(1,0)$ and so on.

To every billiard geodesic in $H$ corresponds a straight line in $\mathbb{R}^2$. Indeed, instead of bouncing on an edge, a straight line is continued in a reflected hexagon along the edge.

\definecolor{ttffqq}{rgb}{0.2,1,0}
\definecolor{wwffqq}{rgb}{0.4,1,0}
\definecolor{fftttt}{rgb}{1,0.2,0.2}
\definecolor{qqqqff}{rgb}{0,0,1}
\definecolor{uququq}{rgb}{0.25,0.25,0.25}
\definecolor{ttccff}{rgb}{0.6,0.6,1}
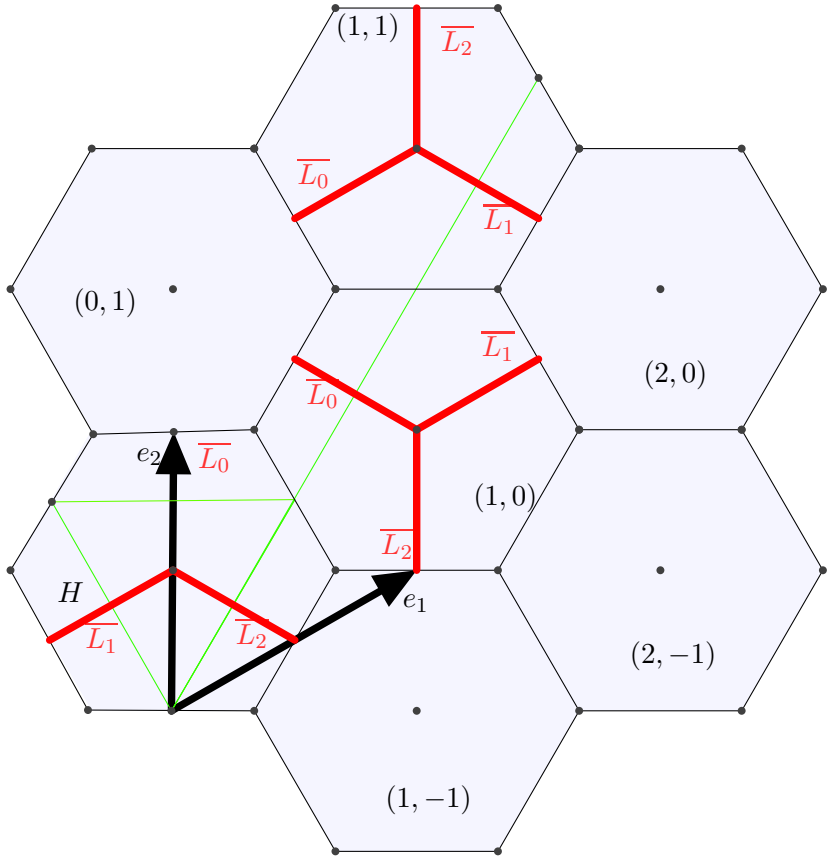
\begin{figure}[!h]
\begin{tikzpicture}[line cap=round,line join=round,>=triangle 45,x=1.0cm,y=1.0cm]
\clip(-1.32,-4.11) rectangle (19.47,8.28);
\fill[line width=2.8pt,color=ttccff,fill=ttccff,fill opacity=0.1] (1.34,-1.64) -- (3.5,-1.64) -- (4.58,0.23) -- (3.5,2.1) -- (1.34,2.1) -- (0.26,0.23) -- cycle;
\fill[line width=2.8pt,color=ttccff,fill=ttccff,fill opacity=0.1] (3.5,2.1) -- (4.58,0.23) -- (6.74,0.23) -- (7.82,2.1) -- (6.74,3.97) -- (4.58,3.97) -- cycle;
\fill[color=ttccff,fill=ttccff,fill opacity=0.1] (1.34,2.1) -- (3.5,2.1) -- (4.58,3.97) -- (3.5,5.84) -- (1.34,5.84) -- (0.26,3.97) -- cycle;
\fill[line width=2.8pt,color=ttccff,fill=ttccff,fill opacity=0.1] (6.74,0.23) -- (4.58,0.23) -- (3.5,-1.64) -- (4.58,-3.51) -- (6.74,-3.51) -- (7.82,-1.64) -- cycle;
\fill[line width=2.8pt,color=ttccff,fill=ttccff,fill opacity=0.1] (4.58,3.97) -- (6.74,3.97) -- (7.82,5.84) -- (6.74,7.71) -- (4.58,7.71) -- (3.5,5.84) -- cycle;
\fill[color=ttccff,fill=ttccff,fill opacity=0.1] (6.74,3.97) -- (7.82,2.1) -- (9.98,2.1) -- (11.06,3.97) -- (9.98,5.84) -- (7.82,5.84) -- cycle;
\fill[color=ttccff,fill=ttccff,fill opacity=0.1] (6.74,0.23) -- (7.82,-1.64) -- (9.98,-1.64) -- (11.06,0.23) -- (9.98,2.1) -- (7.82,2.1) -- cycle;
\draw [line width=1.2pt,color=fftttt] (2.42,0.23)-- (2.4,-1.64);
\draw [->,line width=2.8pt] (2.4,-1.64) -- (5.66,0.23);
\draw [->,line width=2.8pt] (2.4,-1.64) -- (2.43,2.07);
\draw [color=wwffqq] (2.4,-1.64)-- (4.04,1.17)-- (0.81,1.14);
\draw [color=ttffqq] (0.81,1.14)-- (2.4,-1.64)-- (4.04,1.17);
\draw [color=ttffqq] (2.4,-1.64)-- (7.28,6.78);
\draw (1.34,5.84)-- (0.26,3.97);
\draw (0.26,3.97)-- (1.36,2.04);
\draw (1.36,2.04)-- (3.5,2.1);
\draw (3.5,2.1)-- (4.58,3.97);
\draw (4.58,3.97)-- (3.5,5.84);
\draw (1.34,5.84)-- (3.5,5.84);
\draw (3.5,5.84)-- (4.58,7.71);
\draw (4.58,7.71)-- (6.74,7.71);
\draw (4.58,3.97)-- (6.74,3.97);
\draw (7.82,5.84)-- (6.74,7.71);
\draw (7.82,5.84)-- (6.74,3.97);
\draw (7.82,5.84)-- (9.98,5.84);
\draw (11.06,3.97)-- (9.98,5.84);
\draw (9.98,2.1)-- (11.06,3.97);
\draw (7.82,2.1)-- (9.98,2.1);
\draw (6.74,3.97)-- (7.82,2.1);
\draw (6.74,0.23)-- (7.82,2.1);
\draw (4.58,0.23)-- (6.74,0.23);
\draw (3.5,2.1)-- (4.58,0.23);
\draw (1.36,2.04)-- (0.26,0.23);
\draw (1.29,-1.63)-- (0.26,0.23);
\draw (3.5,-1.64)-- (1.29,-1.63);
\draw (4.58,0.23)-- (3.5,-1.64);
\draw (4.58,-3.51)-- (3.5,-1.64);
\draw (6.74,-3.51)-- (4.58,-3.51);
\draw (7.82,-1.64)-- (6.74,-3.51);
\draw (6.74,0.23)-- (7.82,-1.64);
\draw (9.98,-1.64)-- (7.82,-1.64);
\draw (11.06,0.23)-- (9.98,-1.64);
\draw (9.98,2.1)-- (11.06,0.23);
\draw (5.34,0.04) node[anchor=north west] {$ e_1 $};
\draw (1.8,1.99) node[anchor=north west] {$ e_2 $};
\draw (0.75,0.2) node[anchor=north west] {$ H $};
\draw [color=fftttt](2.62,2.07) node[anchor=north west] {$ \overline{L_0} $};
\draw [color=fftttt](4.06,2.92) node[anchor=north west] {$ \overline{L_0} $};
\draw [color=fftttt](3.94,5.85) node[anchor=north west] {$ \overline{L_0} $};
\draw [color=fftttt](1.12,-0.33) node[anchor=north west] {$ \overline{L_1} $};
\draw [color=fftttt](6.39,3.55) node[anchor=north west] {$ \overline{L_1} $};
\draw [color=fftttt](6.41,5.24) node[anchor=north west] {$ \overline{L_1} $};
\draw [color=fftttt](3.12,-0.27) node[anchor=north west] {$ \overline{L_2} $};
\draw [color=fftttt](5.05,0.92) node[anchor=north west] {$ \overline{L_2} $};
\draw [color=fftttt](5.85,7.62) node[anchor=north west] {$ \overline{L_2} $};
\draw (0.96,4.11) node[anchor=north west] {$ (0,1) $};
\draw (6.28,1.52) node[anchor=north west] {$ (1,0) $};
\draw (4.45,7.77) node[anchor=north west] {$ (1,1) $};
\draw (5.11,-2.49) node[anchor=north west] {$(1,-1)$};
\draw (8.54,3.16) node[anchor=north west] {$ (2,0) $};
\draw (8.36,-0.58) node[anchor=north west] {$ (2,-1) $};
\draw [line width=2.8pt,color=ffqqqq] (2.42,0.23)-- (0.78,-0.7);
\draw [line width=2.8pt,color=ffqqqq] (2.42,0.23)-- (4.04,-0.7);
\draw [line width=2.8pt,color=ffqqqq] (5.66,2.1)-- (4.04,3.04);
\draw [line width=2.8pt,color=ffqqqq] (5.66,2.1)-- (5.66,0.23);
\draw [line width=2.8pt,color=ffqqqq] (5.66,2.1)-- (7.28,3.04);
\draw [line width=2.8pt,color=ffqqqq] (5.66,5.84)-- (4.04,4.91);
\draw [line width=2.8pt,color=ffqqqq] (5.66,5.84)-- (7.28,4.91);
\draw [line width=2.8pt,color=ffqqqq] (5.66,5.84)-- (5.66,7.71);
\begin{scriptsize}
\fill [color=uququq] (4.58,0.23) circle (1.5pt);
\fill [color=uququq] (3.5,2.1) circle (1.5pt);
\fill [color=uququq] (0.26,0.23) circle (1.5pt);
\fill [color=uququq] (6.74,0.23) circle (1.5pt);
\fill [color=uququq] (7.82,2.1) circle (1.5pt);
\fill [color=uququq] (6.74,3.97) circle (1.5pt);
\fill [color=uququq] (4.58,3.97) circle (1.5pt);
\fill [color=uququq] (4.58,3.97) circle (1.5pt);
\fill [color=uququq] (3.5,5.84) circle (1.5pt);
\fill [color=uququq] (1.34,5.84) circle (1.5pt);
\fill [color=uququq] (0.26,3.97) circle (1.5pt);
\fill [color=uququq] (3.5,-1.64) circle (1.5pt);
\fill [color=uququq] (4.58,-3.51) circle (1.5pt);
\fill [color=uququq] (6.74,-3.51) circle (1.5pt);
\fill [color=uququq] (7.82,-1.64) circle (1.5pt);
\fill [color=uququq] (7.82,5.84) circle (1.5pt);
\fill [color=uququq] (6.74,7.71) circle (1.5pt);
\fill [color=uququq] (4.58,7.71) circle (1.5pt);
\fill [color=uququq] (3.5,5.84) circle (1.5pt);
\fill [color=uququq] (9.98,2.1) circle (1.5pt);
\fill [color=uququq] (11.06,3.97) circle (1.5pt);
\fill [color=uququq] (9.98,5.84) circle (1.5pt);
\fill [color=uququq] (7.82,5.84) circle (1.5pt);
\fill [color=uququq] (9.98,-1.64) circle (1.5pt);
\fill [color=uququq] (11.06,0.23) circle (1.5pt);
\fill [color=uququq] (9.98,2.1) circle (1.5pt);
\fill [color=uququq] (7.82,2.1) circle (1.5pt);
\fill [color=uququq] (2.42,0.23) circle (1.5pt);
\fill [color=uququq] (2.42,3.97) circle (1.5pt);
\fill [color=uququq] (5.66,5.84) circle (1.5pt);
\fill [color=uququq] (8.9,3.97) circle (1.5pt);
\fill [color=uququq] (8.9,0.23) circle (1.5pt);
\fill [color=uququq] (5.66,-1.64) circle (1.5pt);
\fill [color=uququq] (5.66,2.1) circle (1.5pt);
\fill [color=uququq] (1.29,-1.63) circle (1.5pt);
\fill [color=uququq] (2.4,-1.64) circle (1.5pt);
\fill [color=uququq] (1.36,2.04) circle (1.5pt);
\fill [color=uququq] (0.81,1.14) circle (1.5pt);
\fill [color=uququq] (2.43,2.07) circle (1.5pt);
\fill [color=uququq] (7.28,6.78) circle (1.5pt);
\end{scriptsize}
\end{tikzpicture}
\caption{Tiling by hexagons to unfold a billiard trajectory.}
\label{fig:tiling}
\end{figure}

In this description, periodic trajectories correspond to straight lines following a direction
\[v_{a,b}=ae_1+be_2,\,\,a,b\in\mathbb{Z}.\]
For the systole, we are only intersted in simple closed geodesics, so we may assume that $\gcd(a,b)=1$. The actual periodic geodesic corresponds to the curve 
\[\fonction{\gamma_{a,b,x}}{\crochet{0,r_{a,b,x}}}{\mathbb{R}^2}{t}{\begin{pmatrix}
x\\0
\end{pmatrix}+tv_{a,b}},\]
where $r_{a,b,x}>1$ is a rational number. Note that the periodic geodesic depends on the starting point $\begin{pmatrix}
x\\0
\end{pmatrix}$.

The fact that we need to allow the curve to go beyond the hexagon $(a,b)$ supports two situations. First, when the trajectory reaches $(a,b)$, the hexagon $(a,b)$ might not be an image by a translation of the original hexagon $H$, hence the edge we reach might not correspond by the action of $\mathbb{Z}e_1\oplus\mathbb{Z}e_2$ on $\mathbb{R}^2$ to the original starting edge. And second, not every periodic billiard trajectory gives a periodic geodesic on $(\mathbb{S}^2,d_H)$, it is only the case if the trajectory meets an even number of hexagons. This number $r_{a,b,x}$ may depend on the starting point $x$.

For instance, we will argue that the systole is achieved by $a=0, b=1$, for which $r_{0,1,x}$ has to be $2$ for any $x$. See figure \ref{fig:tiling}.

In the following discussion, we will use repeatedly the computational arguments below.

\begin{remark}
\begin{enumerate}
\item Any straight line that joins an edge of an hexagon to the opposite edge has length 2. For further use, we call those lines generalized diameters.
\item Any straight line that joins an edge which does not belong to a $l^1$-type triangle to the parallel diagonal inside the hexagon has length 1. We call such a line a generalized radius. Such a triangle contains one of the segment $\overline{L_i}$ and we say that the corresponding radius is associated to the the segment $\overline{L_i}$.
\item In the computation of the systole, we will frequently encounter the following situation that we describe now once and for all. We consider a diamond made with two triangles of type $l^1$ glued along an edge shared by two hexagons $(a,b)$ and $(a',b')$. If a geodesic crosses the hexagon entering and exiting by two opposite sides, there is segment $\overline{L_i}$ in $(a,b)$ and a segment $\overline{L_{i'}}$ in $(a',b')$ with the property that the sum of the lengths of the projection of the part of the geodesic in $(a,b)$ on $\overline{L_i}$ and the length of the projection of the part of the geodesic in $(a',b')$ on $\overline{L_{i'}}$ is 1. See figure \ref{fig:diamond}.
We will refer to this situation by saying that the geodesic crosses the diamond $(a,b),(a',b')$ (if necessary: with contributing lengths on the pair $\overline{L_i}$, $\overline{L_{i'}}$
).\\

\definecolor{wwffww}{rgb}{0.4,1,0.4}
\definecolor{ffqqqq}{rgb}{1,0,0}
\definecolor{ttccff}{rgb}{0.6,0.6,1}
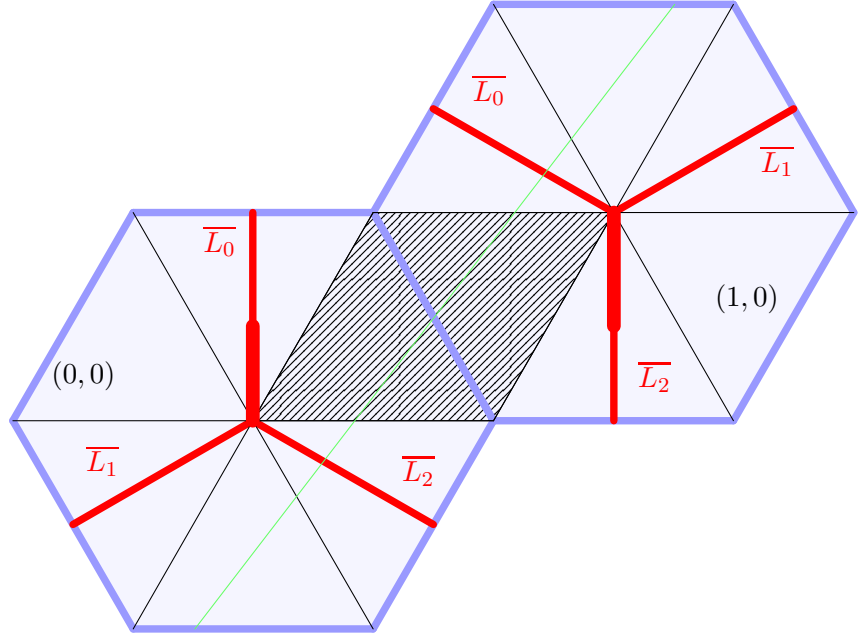
\begin{figure}[!h]
\begin{tikzpicture}[line cap=round,line join=round,>=triangle 45,x=1.0cm,y=1.0cm]
\clip(-1.6,-0.78) rectangle (14.52,9.67);
\fill[line width=2.8pt,color=ttccff,fill=ttccff,fill opacity=0.1] (2.17,0.25) -- (5.36,0.25) -- (6.96,3.02) -- (5.36,5.79) -- (2.17,5.79) -- (0.57,3.02) -- cycle;
\fill[line width=2.8pt,color=ttccff,fill=ttccff,fill opacity=0.1] (5.36,5.79) -- (6.96,3.02) -- (10.15,3.02) -- (11.75,5.79) -- (10.15,8.56) -- (6.96,8.56) -- cycle;
\fill[fill=black,pattern=north east lines] (3.76,3.02) -- (6.96,3.02) -- (8.56,5.79) -- (5.36,5.79) -- cycle;
\draw [line width=2.8pt,color=ttccff] (2.17,0.25)-- (5.36,0.25);
\draw [line width=2.8pt,color=ttccff] (5.36,0.25)-- (6.96,3.02);
\draw [line width=2.8pt,color=ttccff] (6.96,3.02)-- (5.36,5.79);
\draw [line width=2.8pt,color=ttccff] (5.36,5.79)-- (2.17,5.79);
\draw [line width=2.8pt,color=ttccff] (2.17,5.79)-- (0.57,3.02);
\draw [line width=2.8pt,color=ttccff] (0.57,3.02)-- (2.17,0.25);
\draw [line width=2.8pt,color=ttccff] (5.36,5.79)-- (6.96,3.02);
\draw [line width=2.8pt,color=ttccff] (6.96,3.02)-- (10.15,3.02);
\draw [line width=2.8pt,color=ttccff] (10.15,3.02)-- (11.75,5.79);
\draw [line width=2.8pt,color=ttccff] (11.75,5.79)-- (10.15,8.56);
\draw [line width=2.8pt,color=ttccff] (10.15,8.56)-- (6.96,8.56);
\draw [line width=2.8pt,color=ttccff] (6.96,8.56)-- (5.36,5.79);
\draw [line width=2.8pt,color=ffqqqq] (3.76,3.02)-- (1.37,1.64);
\draw [line width=2.8pt,color=ffqqqq] (3.76,3.02)-- (3.76,5.79);
\draw [line width=2.8pt,color=ffqqqq] (3.76,3.02)-- (6.16,1.64);
\draw [line width=2.8pt,color=ffqqqq] (8.56,5.79)-- (6.16,7.17);
\draw [line width=2.8pt,color=ffqqqq] (8.56,5.79)-- (8.56,3.02);
\draw [line width=2.8pt,color=ffqqqq] (8.56,5.79)-- (10.95,7.17);
\draw (3.76,3.02)-- (2.17,5.79);
\draw (3.76,3.02)-- (5.36,5.79);
\draw (3.76,3.02)-- (6.96,3.02);
\draw (3.76,3.02)-- (5.36,0.25);
\draw (3.76,3.02)-- (2.17,0.25);
\draw (3.76,3.02)-- (0.57,3.02);
\draw (8.56,5.79)-- (5.36,5.79);
\draw (8.56,5.79)-- (6.96,3.02);
\draw (8.56,5.79)-- (10.15,3.02);
\draw (8.56,5.79)-- (11.75,5.79);
\draw (8.56,5.79)-- (10.15,8.56);
\draw (8.56,5.79)-- (6.96,8.56);
\draw (3.76,3.02)-- (6.96,3.02);
\draw (6.96,3.02)-- (8.56,5.79);
\draw (8.56,5.79)-- (5.36,5.79);
\draw (5.36,5.79)-- (3.76,3.02);
\draw [color=wwffww] (2.99,0.25)-- (9.37,8.56);
\draw [line width=5.2pt,color=ffqqqq] (3.76,4.28)-- (3.76,3.02);
\draw [line width=5.2pt,color=ffqqqq] (8.56,4.28)-- (8.56,5.79);
\draw (0.95,3.95) node[anchor=north west] {$ (0,0) $};
\draw (9.77,4.97) node[anchor=north west] {$(1,0)$};
\draw [color=ffqqqq](2.97,5.74) node[anchor=north west] {$ \overline{L_0} $};
\draw [color=ffqqqq](1.41,2.82) node[anchor=north west] {$ \overline{L_1} $};
\draw [color=ffqqqq](5.61,2.69) node[anchor=north west] {$ \overline{L_2} $};
\draw [color=ffqqqq](8.75,3.94) node[anchor=north west] {$\overline{L_2} $};
\draw [color=ffqqqq](6.55,7.75) node[anchor=north west] {$ \overline{L_0}$};
\draw [color=ffqqqq](10.38,6.81) node[anchor=north west] {$ \overline{L_1} $};
\end{tikzpicture}
\caption{The geodesic crosses a diamond $(0,0)$ $(1,0)$ with contributing lengths in $\overline{L_0}$, $\overline{L_2}$.}
\label{fig:diamond}
\end{figure}

\end{enumerate}
\end{remark} 


To compute the systole, we have to bound the length from below to every periodic geodesic. Note that, to a trajectory $\gamma_{a,b,x}$ directed by $v_{a,b}$, we associate the sequence of hexagons $(0,0),\cdots,(a,b),\cdots$ visited during the trajectory. This sequence is well defined since every trajectory intersects the edges transversaly (after we excluded trajectories hitting the corners of the hexagon) but may depend on $x$ and not only on the direction $v_{a,b}$.

For further description of the geodesics, we assume without loss of generality that in the original hexagon $H$, the segements $\overline{L_i}$'s are organized such that $\overline{L_0}$ is directed by the $y$-axis and $\overline{L_0}, \overline{L_1}, \overline{L_2}$ are cyclically ordered (such as in figure \ref{fig:tiling}). The geometry of each hexagon along a trajectory is deduced accordingly.

By symmetry with respect to $e_2$, we may assume that $a\geqslant 0$. That forces every periodic trajectory to fall into one of the following types according to the three first hexagons visited. 

\vspace{0.5cm}

\textbf{Case 1 : Trajectory of type $(0,0) (0,1) (0,2)\cdots (a,b)$.}

Such a trajectory covers two generalized diameters. It follows that it has length greater or equal than 4 (equality holds, e.g. for $\gamma_{0,1,x}$ for any $x$).

\vspace{0.5cm}

\textbf{Case 2 : Trajectory of type $(0,0) (0,1) (1,1)\cdots (a,b)$.}

In this case $(a,b)$ cannot be $(1,1)$ because the hexagons visited by $\gamma_{1,1,x}$
are $(0,0) (1,0) (1,1)$ for every $x$. Also the trajectory $(0,0) (0,1)$ is not closed since $r_{0,1,x}=2$ for any $x$.

When such a geodesic reaches the hexagon $(1,1)$, it has covered a generalized diameter in $(0,0)$ and 2 generalized radius in $(1,1)$, associated to $\overline{L_0}$ and $\overline{L_2}$. Hence, its length must be at least 4.

\vspace{0.5cm}

\textbf{Case 3 : Trajectory of type $(0,0) (1,0) (1,1)\cdots (a,b)$.}

In this case $(a,b)$ cannot be $(1,1)$ because $r_{1,1,x}=3$ (in particular the trajectory does not close on the edge between $(1,1)$ and $(2,1)$ because there is an odd number of hexagons visited). Hence we can further split the trajectory into different subcases.

In all of the different subcases, the trajectory covers one diamond of type $(0,0)$ $(1,0)$ with contribution by $\overline{L_0}$ and $\overline{L_2}$, delivering already length 1.

\paragraph*{Subcase A} Trajectory of type $(0,0) (1,0) (1,1) (1,2)\cdots (a,b)$.

The trajectory must cover at least one generalized diameter inside $(1,1)$ for an additional length of 2. Also, any such geodesic will continue either inside $(1,3)$ or $(2,3)$ and must cover a generalized radius associated to $\overline{L_2}$ in $(1,2)$. We get a length bigger than 4.

\paragraph*{Subcase B} Trajectory of type $(0,0) (1,0) (1,1) (2,1)\cdots (a,b)$.

Beside the first diamond, such a trajectory must cross the diamond $(1,0),(1,1)$ with contribution in $\overline{L_1}$ and $\overline{L_0}$ and the diamond $(1,1),(2,1)$ with contribution in $\overline{L_2}$ and $\overline{L_1}$.

So far, the length of the geodesic is at least 3. If it follows by entering the hexagon $(3,1)$, it covers another generalized radius of type $\overline{L_0}$ in $(2,1)$. If it follows by the hexagon $(2,2)$, it crosses another diamond $(2,1),(2,2)$ with contribution in $\overline{L_2}$ and $\overline{L_0}$.

In any case, the length is then at least 4.

\paragraph*{Subcase C} Trajectory of type $(0,0) (1,0) (1,1) (2,0)\cdots (a,b)$.

Let us first assume that the trajectory is not some $\gamma_{2,1,x}$ with $r_{2,1,x}=1$, which means that the trajectory does not stop entering $(2,1)$.

Such a trajectory must continue by $(3,1)$. Hence it crosses 4 diamonds between $(0,0)$ and $(1,0)$, between $(1,0)$ and $(1,1)$, between $(2,0)$ and $(2,1)$ and between $(2,1)$ and $(3,1)$ giving at least length 4.

It remains to compute the length of the closed geodesics $\gamma_{2,1,x}$ for which $r_{2,1,x}=1$. Beside the two diamonds between $(0,0)$ and $(1,0)$ and between $(1,0)$ and $(1,1)$, there is a "split diamond" $(0,0)$ $(2,0)$ with contribution in $\overline{L_1}$ and $\overline{L_2}$.

The rest of the contribution to the length is given by $\overline{L_2}$ in $(0,0)$, $\overline{L_0}$ in $(1,0)$ and $\overline{L_1}$ in $(1,1)$ and $(2,0)$. The sum of those contributions is exactly 1 and $\gamma_{2,1,x}$ is a systole. The union of the projections of the parts of the geodesic inside each heaxagon cover a complete segment.

\vspace{0.5cm}

\textbf{Case 4 : Trajectory of type $(0,0) (1,0) (2,0)\cdots (a,b)$.}

This trajectory covers a generalized diameter in $(1,0)$ and either another generalized diameter in $(2,0)$ if it follows by entering $(3,0)$, either two generalized radii in $(2,0)$, associated to $\overline{L_0}$ and $\overline{L_1}$ if it goes toward $(2,1)$.

\vspace{0.5cm}

\textbf{Case 5 : Trajectory of type $(0,0) (1,-1) (2,-1)\cdots (a,b)$.}

Such a trajectory must contain two generalized radii associated to $\overline{L_0}$ and $\overline{L_2}$ in $(1,-1)$ for a length of 2.

Then, if the trajectory crosses $(2,-1)$ escaping through $(2,0)$, it must cover two additional generalized radii in $(2,-1)$, associated to $\overline{L_0}$ and $\overline{L_2}$. If the trajectory escapes through $(3,-1)$, it covers a generalized diameter inside $(2,-1)$. In any case, we get at least length 2 in $(2,-1)$.

\vspace{0.5cm}

\textbf{Case 6 : Trajectory of type $(0,0) (1,-1) (1,0)\cdots (a,b)$.}


\paragraph*{Subcase A} Trajectory of type $(0,0) (1,-1) (1,0) (2,0)\cdots (a,b)$.

Such a trajectory continues necessarily by $(3,0)$ and must crosses two diamonds between $(1,-1)$ and $(1,0)$ and between $(1,0)$ and $(2,0)$, the generalized radii associated to $\overline{L_2}$ in $(2,0)$ and the one associated to $\overline{L_2}$ in $(3,0)$. We get a length at least 4.

\paragraph*{Subcase B} Trajectory of type $(0,0) (1,-1) (1,0) (2,-1)\cdots (a,b)$.

Such a trajectory must continue by $(2,1)$ and may stop after crossing $(2,1)$. Along the way, it crosses two diamonds $(1,-1)$ $(1,0)$ and $(2,-1)$, $(2,0)$, and one split diamond $(0,0)$ $(2,0)$ with contribution in $\overline{L_1}$ and $\overline{L_2}$.

A similar argument to the diamond delivers another length 1. Indeed the geodesic enters $(1,-1)$ by the edge orthogonal to $\overline{L_2}$ and exits $(1,0)$ by the edge orthogonal to $\overline{L_0}$. Hence, the sum of the length of the projection of the part of the geodesic inside $(1,-1)$ on $\overline{L_2}$ and the part inside $(1,0)$ on $\overline{L_0}$ equals 1.
\end{proof}

Theorem \ref{thm:main} then follows from Proposition \ref{prop:metricproperties} (2) and Proposition \ref{prop:systole}.

\section{Regularity Issues}\label{sec:regularity}

So far the metric on $\mathbb{S}^2$ that we constructed has singularities and is only smooth outside of the glueing edges. We conclude this paper by sketching how we could improve regularity along the edges without changing the geodesic flow. This discussion is based on \cite{pogorelovbookhilbert4} and \cite[Chapter 10]{cossarinisabourau}. The goal of this section is to check that we can extend the regularization procedure in \cite[Paragraph 11.2]{cossarinisabourau} from the hexagon to the sphere. We do not repeat the proofs which are straightforward adaptations.

It is remarkable that the symplectic form described in Section \ref{sec:preliminaries}, in the special case of projective Finsler structure on $\mathbb{R}^2$, characterizes the metric, as well as its regularity. Indeed, consider the set of oriented lines $\Gamma$ in $\mathbb{R}^2$, identified with $\mathbb{R}\times \mathbb{S}^1$ (the factor $\mathbb{S}^1$ is the direction, the factor $\mathbb{R}$ gives the signed distance to the origin). Let $\mu$ be a nonnegative Borel measure on $\Gamma$. We assume that the measure is invariant under reversing the orientation, gives finite mass to every compact set, gives zero mass on the subset of lines passing through a given point and gives positive mass to the set of lines crossing a given non degenerate segment. To such a measure, we associate the distance
\[d_\mu(x,y)=\frac{1}{4}\int_\Gamma \sharp \left(\gamma\cap \left[x,y\right]\right)d\mu(\gamma).\]
This distance is projective: straight lines, suitably reparametrized, are geodesics.

\begin{theorem}{\cite{pogorelovbookhilbert4} and \cite[Chapter 10]{cossarinisabourau}}\label{thm:regularity}
The distance $d_\mu$ is Finsler if and only if $\mu$ is a positive smooth measure, in which case it coincides with the symplectic measure.
\end{theorem}

The usefullness of this statement comes from the fact that the distance $d_H$ on the hexagon is associated to a measure and we can regularize the metric by regularizing the measure while preserving the geodesic flow (every metric is projective).

By performing a stereographic projection through one of the singularities, we may replace the sphere $\mathbb{S}^2$ by $\mathbb{R}^2$. We identify each of the hexagons with their images under stereographic projection. One of the hexagon is denoted $H_d$ and the other one $H_u$. We consider the family of lines $D_{k,d}^t$ and $D_{k,u}^t$ where $D_{k,d}^t$ is the line inside $H_d$ orthogonal to $\overline{L_k}$ at distance $t$ from the origin (similarly in $H_u$). For a given $t$ and $k$, $D_{k,d}^t$ consists in two oriented lines In the coordinate system $\Gamma=\mathbb{R}\times \mathbb{S}^1$, the set $D_{0,d}^t\cup D_{1,d}^t\cup D_{2,d}^t$ is written as $\left[0,1\right]\times\left\lbrace e^{i\frac{\pi}{6}},e^{i\frac{5\pi}{6}},e^{i\frac{3\pi}{2}}\right\rbrace\cup \left[-1,0\right]\times\left\lbrace e^{i\frac{7\pi}{6}},e^{i\frac{11\pi}{6}},e^{i\frac{\pi}{2}}\right\rbrace$ (the union corresponds to the two possible orientations). We endow $\mathbb{R}\times \mathbb{S}^1$ with the usual Lebesgue measure $dxd\theta$ ($d\theta$ has finite mass $2\pi$). The measure $\mu_d$ is (2 times) the restriction of $dxd\theta$ to $D_{0,d}^t\cup D_{1,d}^t\cup D_{2,d}^t$. We construct similarly $\mu_u$. Finally we set
\[\mu=\mu_d+\mu_u.\]
We now conclude with the very same procedure as in \cite[Paragraph 11.2]{cossarinisabourau}. We convolve the measure $\mu +  \varepsilon dxd\theta$ with a smooth approximation of the Dirac mass. With theorem \ref{thm:regularity} and \cite[Lemma 11.5]{cossarinisabourau}, this measure gives in turn a smooth Finsler metric $\varepsilon$-close to $(\mathbb{S}^2,d_H)$.

A much more delicate task is to remove the singularities.

\bibliographystyle{alpha}

\bibliography{biblio}

\end{document}